\documentclass[11pt,letterpaper]{amsart}
\usepackage{amsthm}
\usepackage{amsmath}
\usepackage{amssymb}
\usepackage{comment}
\usepackage{tikz}
\usepackage{enumitem}
\usepackage{amsfonts}
\usepackage{mathtools}
\usepackage{hyperref}
\usepackage{todonotes}
\usepackage[utf8]{inputenc}

\DeclareMathOperator{\fl}{fl}
\DeclareMathOperator{\mh}{mHeight}
\DeclareMathOperator{\mcontent}{mCont}
\DeclareMathOperator{\supp}{supp}
\DeclareMathOperator{\inv}{Inv}
\DeclareMathOperator{\Lie}{Lie}
\DeclareMathOperator{\spn}{span}
\DeclareMathOperator{\mg}{mag}
\DeclareMathOperator{\sgn}{sign}

\newcommand{\ph}{\varphi}

\newcommand{\bs}{\boldsymbol}

\newcommand{\e}{\varepsilon}

\newcommand{\SL}{\mathrm{SL}}
\newcommand{\SO}{\mathrm{SO}}
\newcommand{\R}{\mathbb{R}}
\newcommand{\Z}{\mathbb{Z}}

\renewcommand{\S}{\mathfrak{S}}
\newcommand{\D}{\mathfrak{D}}

\newcommand\precdot{\mathrel{\ooalign{$\prec$\cr
  \hidewidth\raise0.001ex\hbox{$\cdot\mkern0.6mu$}\cr}}}

\newtheorem{theorem}{Theorem}[section]
\newtheorem{def-prop}[theorem]{Definition-Proposition}
\newtheorem{prop}[theorem]{Proposition}
\newtheorem{conj}[theorem]{Conjecture}
\newtheorem{lemma}[theorem]{Lemma}
\newtheorem{claim}[theorem]{Claim}

\theoremstyle{definition}
\newtheorem{ex}[theorem]{Example}

\newtheorem{defin}[theorem]{Definition}

\theoremstyle{remark}
\newtheorem*{remark}{Remark}

\begin{document}
\title{On the minimal power of $q$ in a Kazhdan--Lusztig polynomial}

\author{Christian Gaetz}
\address[Gaetz]{Department of Mathematics, University of California, Berkeley, CA, USA.}
\email{\href{mailto:gaetz@berkeley.edu}{{\tt gaetz@berkeley.edu}}}

\author{Yibo Gao}
\address[Gao]{Beijing International Center for Mathematical Research, Peking University, Beijing, China.}
\email{\href{mailto:gaoyibo@bicmr.pku.edu.cn}{{\tt gaoyibo@bicmr.pku.edu.cn}}}

\date{\today}

\begin{abstract}
For $w$ in the symmetric group, we provide an exact formula for the smallest positive power $q^{h(w)}$ appearing in the Kazhdan--Lusztig polynomial $P_{e,w}(q)$. We also provide a tight upper bound on $h(w)$ in simply-laced types, resolving a conjecture of Billey--Postnikov from 2002.
\end{abstract}
\maketitle

\section{Introduction}
\label{sec:intro}
Let $G$ be a simply connected semisimple complex Lie group, with Borel subgroup $B$ containing maximal torus $T$ and corresponding Weyl group $W$. The Bruhat decomposition $G=\bigsqcup_{w \in W} BwB$ gives rise to the \emph{Schubert varieties} $X_w \coloneqq \overline{BwB/B}$ inside the flag variety $G/B$, whose containments determine the Bruhat order on $W$ ($y \leq w$ if $X_y \subset X_w$). The \emph{Kazhdan--Lusztig polynomials} $P_{y,w}(q) \in \mathbb{Z}[q]$ have since their discovery \cite{Kazhdan-Lusztig-1} proven to underlie deep connections between canonical bases of Hecke algebras, singularities of Schubert varieties \cite{Kazhdan-Lusztig-2}, and representations of Lie algebras \cite{Beilinson-Bernstein, Brylinski-Kashiwara}.

\begin{theorem}[Kazhdan and Lusztig \cite{Kazhdan-Lusztig-2}]
\label{thm:IH-interpretation-of-kl-polys}
For $y \leq w$, let $IH^*(X_w)_y$ denote the local intersection cohomology of $X_w$ at the $T$-fixed point $yB$, then 
\[ P_{y,w}(q) = \sum_i \dim(IH^{2i}(X_w)_y) q^i.\]
\end{theorem}

Theorem~\ref{thm:IH-interpretation-of-kl-polys} implies that $P_{y,w}(q)$ has nonnegative coefficients, a property which is completely obscured by their recursive definition (see Definition~\ref{def:kl-polys}); this was proven for arbitrary Coxeter groups $W$ by Elias and Williamson \cite{Elias-Williamson}. It is known that for all $y \leq w$ one has $P_{y,w}(0)=1$.

\begin{theorem}[Deodhar \cite{Deodhar-smooth}; Peterson (see \cite{Carrell-Kuttler-1})]
\label{thm:smooth-iff-kl-poly-trivial}
Suppose $G$ is simply-laced and $y \leq w$, then $X_w$ is smooth at $yB$ if and only if $P_{y,w}(q)=1$. In particular, $X_w$ is a smooth variety if and only if $P_{e,w}(q)=1$.
\end{theorem}

In light of Theorem~\ref{thm:IH-interpretation-of-kl-polys}, one would like to understand $P_{y,w}(q)$ explicitly enough to determine which coefficients vanish. Indeed, the view of the $P_{y,w}$ as a measure of the failure of local Poincar\'{e} duality in $X_w$ was among the original motivations for their introduction \cite{Kazhdan-Lusztig-1}. Unfortunately $P_{y,w}$ may be arbitrarily complicated \cite{Polo} and the explicit formulae \cite{Brenti-combinatorial-formula} which exist involve cancellation, and are thus not well-suited to this problem. If $X_w$ is singular (as is true generically) one could at least ask for the smallest nontrivial coefficient, the first degree in which Poincar\'{e} duality fails. Writing $[q^i]P_{y,w}$ for the coefficient of $q^i$ in $P_{y,w}(q)$, define:
\begin{align*}
    h(w) &\coloneqq \min \{i>0 \mid [q^i]P_{e,w} \neq 0 \}, \\
    &= \min_{y \leq w} \min \{i>0 \mid [q^i]P_{y,w} \neq 0\}, \\
    &= \min_{y \leq w} \min \{i>0 \mid IH^{2i}(X_w)_y \neq 0\}.
\end{align*}
 The second equality follows from the surjection $IH^*(X_w)_x \twoheadrightarrow IH^*(X_w)_y$ for $x \leq y \leq w$ constructed by Braden and Macpherson \cite{Braden-MacPherson-monotonicity}. We make the convention that $h(w)=+ \infty$ when $X_w$ is smooth.

\begin{conj}[Billey and Postnikov \cite{Billey-Postnikov}]
\label{conj:bp-conjecture}
Let $G$ be simply-laced of rank $r$, and let $w \in W$ such that $X_w \subset G/B$ is singular, then $h(w) \leq r$.
\end{conj}

Billey and Postnikov's conjecture is somewhat surprising, since $\deg(P_{y,w})$ may be as large as $\frac{1}{2}(\ell(w)-\ell(y)-1)$ which is of the order of $r^2$, where $\ell$ denotes Coxeter length. A constant upper bound on $h(w)$ in certain special infinite Coxeter groups was given in \cite{Richmond-Slofstra-triangle-group}.

The decomposition $X_w = \bigsqcup_{y \leq w} ByB/B$ is an affine paving, with the cell $ByB/B$ having complex dimension $\ell(y)$. We thus have
\[
L(w) \coloneqq \sum_{y \leq w} q^{\ell(y)} = \sum_{j \geq 0} \dim(H^{j}(X_w))q^{j/2},
\]
the Poincar\'{e} polynomial of $X_w$. Bj\"{o}rner and Ekedahl \cite{Bjorner-Ekedahl} gave a precise interpretation of $h(w)$ in terms of $L(w)$, as the smallest homological degree in which Poincar\'{e} duality fails.

\begin{theorem}[Bj\"{o}rner and Ekedahl \cite{Bjorner-Ekedahl}]
\label{thm:bjorner-ekedahl}
For $0 \leq i \leq \ell(w)/2$ we have $[q^i]L(w) \leq [q^{\ell(w)-i}]L(w)$, and 
\[
h(w)=\min \{i\geq 0 \mid [q^i]L(w) < [q^{\ell(w)-i}]L(w)\}.
\]
\end{theorem}

Theorem~\ref{thm:bjorner-ekedahl} will be a useful tool in this work, but cannot be directly used to resolve Conjecture~\ref{conj:bp-conjecture} since it is difficult to compute $[q^i]L(w)$ in general.

Our first main theorem\footnote{An extended abstract summarizing part of this paper appears in the proceedings of FPSAC 2024 \cite{fpsac}.} is a refinement and proof of Conjecture~\ref{conj:bp-conjecture}.
\begin{theorem}
\label{thm:simply-laced-upper-bound}
Let $G$ be simply-laced of rank $r$, and let $w \in W$ such that $X_w \subset G/B$ is singular, then $h(w) \leq r-2$.
\end{theorem}
The bound of $r-2$ is tight when $G$ is a member of the infinite families $\SL_{r+1}$ or $\SO_{2r}$. 

When $G$ is one of the exceptional simply-laced groups of type $E_6,E_7,$ or $E_8$, Theorem~\ref{thm:simply-laced-upper-bound} follows from the computations made by Billey--Postnikov \cite{Billey-Postnikov}. In the case $G=\SL_{n+1}$, the theorem can be derived from the classification of the singular locus of $X_w$ \cite{Billey-Warrington, Cortez-1, Cortez-2, Kassel-Lascoux-Reutenauer, Manivel}. However, in this case we provide a new exact formula for $h(w)$ for any permutation $w$ ($W$ is isomorphic to the symmetric group $\S_{n+1}$). This theorem is phrased in terms of \emph{pattern containment} (see Section~\ref{sec:pattern-conventions}).

\begin{theorem}
\label{thm:type-A-formula}
Let $G=\SL_{n+1}$, and let $w \in W=\S_{n+1}$ such that $X_w \subset G/B$ is singular, then 
\[ h(w)=\begin{cases} 1 &\text{ if $w$ contains $4231$,} \\
\mh(w) &\text{ otherwise,} \end{cases}\]
where $\mh(w)$ denotes the minimum height of a $3412$ pattern in $w$. 
\end{theorem}

In the case $P_{e,w}(1)=2$, Theorem~\ref{thm:type-A-formula} follows from the work of Woo \cite{Woo}. Our theorem adds to the deep \cite{Woo-Yong-Gorenstein, Woo-Yong-governing} and ubiquitous \cite{Billey-Abe} links between singularities of Schubert varieties and pattern containment.

\begin{remark}
We thank Alexander Woo for pointing out the following implication of Theorem~\ref{thm:type-A-formula}: the quantity $h(w)$ may be calculated by looking only at the generic singularities of $X_w$. More specifically, by the formulae in \cite{Billey-Warrington}, $h(w)$ can be obtained from knowledge of $P_{v,w}(q)$ for all \emph{maximal} cells $BvB/B$ in the singular locus of $X_w$.
\end{remark}

\section{Preliminaries}
\label{sec:prelim}
\subsection{Bruhat order}
Let $W$ be a Weyl group with simple reflections $S=\{s_1,s_2,\ldots\}$ and length function $\ell$. Write $R$ for the set of reflections (conjugates of simple reflections), then \emph{Bruhat order $\leq$} on $W$ is defined as the transitive closure of the relation $y<yr$ if $r \in R$ and $\ell(y)<\ell(yr)$. 

\begin{theorem}[Deodhar \cite{Deodhar-bruhat-order}]
\label{thm:subword-bruhat}
Let $\bs{s}=s_{i_1}\cdots s_{i_{\ell}}$ be a reduced word for $w$, then $y \leq w$ if and only if some subsequence of $\bs{s}$ is a reduced word for $y$.
\end{theorem}

\subsection{Kazhdan--Lusztig polynomials}
The left (respectively, right) \emph{descents} $D_L(w)$ (resp. $D_R(w)$) are those $s \in S$ such that $sw<w$ (resp. $ws<w$). 

\begin{defin}[Kazhdan and Lusztig \cite{Kazhdan-Lusztig-1}]
\label{def:kl-polys}
Define polynomials $R_{y,w}(q) \in \Z[q]$ by:
\[R_{y,w}(q) = \begin{cases} 0, &\text{ if $y \not \leq w$.} \\ 1, &\text{ if $y=w$.} \\ R_{ys,ws}(q), &\text{ if $s \in D_R(y) \cap D_R(w)$.} \\ qR_{ys,ws}(q) + (q-1)R_{y,ws}, &\text{ if $s \in D_R(w)\setminus D_R(y)$.}\end{cases}\]
Then there is a unique family of polynomials $P_{y,w}(q) \in \Z[q],$ the \emph{Kazhdan--Lusztig polynomials} satisfying $P_{y,w}(q)=0$ if $y \not \leq w$, $P_{w,w}(q)=1$, and such that if $y<w$ then $P_{y,w}$ has degree at most $\frac{1}{2}(\ell(w)-\ell(y)-1)$ and 
\[q^{\ell(w)-\ell(y)}P_{y,w}(q^{-1})=\sum_{a \in [y,w]} R_{y,a}(q)P_{a,w}(q).\]
\end{defin}

Although not immediate from the recursion in Definition~\ref{def:kl-polys}, which seems to privilege right multiplication by $s$ over left multiplication, the following proposition follows from the uniqueness of the canonical basis of the Hecke algebra $\mathcal{H}(W)$ (Theorem 1.1 of \cite{Kazhdan-Lusztig-1}).

\begin{prop}
\label{prop:kl-of-inverse}
Let $y,w \in W$, then $P_{y,w}(q)=P_{y^{-1},w^{-1}}(q)$. In particular, $h(w)=h(w^{-1})$.
\end{prop}
 
\subsection{Fiber bundles of Schubert varieties}

For $J \subset S$, we write $W_J$ for the subgroup generated by $J$, $P_J$ for the parabolic subgroup of $G$ generated by $B$ and $J$, and $W^J$ for the set of minimal length representatives of the left cosets $W/W_J$. We have $W^J=\{w \in W \mid D_R(w) \cap J = \emptyset\}$. Each $w \in W$ decomposes uniquely as $w^Jw_J$ with $w^J \in W^J$ and $w_J \in W_J$.  Using right cosets instead gives decompositions $w=\prescript{}{J}{w} \prescript{J}{}{w}$ with $\prescript{}{J}{w} \in W_J$ and $\prescript{J}{}{w} \in \prescript{J}{}{W} = (W^J)^{-1}$. Notice that $(w^{-1})_J = (\prescript{}{J}{w})^{-1}.$

We write $w_0(J)$ for the unique element of $W_J$ of maximum length and write $[u,v]^J$ for the set $[u,v] \cap W^J$. Since parabolic decompositions are unique, we have an injection $[e,w^J]^J \times [e,w_J] \hookrightarrow [e,w]$ given by multiplication.

Schubert varieties $X_{w^J}^J \coloneqq \overline{Bw^JP_J/P_J}$ in the partial flag variety $G/P_J$ have an affine paving 
\[\bigsqcup_{\substack{y \in W^J \\ y \leq w^J}} ByP_J/P_J,\]
and so
\[L^J(w^J) \coloneqq \sum_{\substack{y \in W^J \\ y \leq w^J}} q^{\ell(y)} = \sum_{j \geq 0} \dim(H^j(X^J_{w^J}))q^{j/2}.
\]

\begin{defin}[Richmond and Slofstra \cite{Richmond-Slofstra-fiber-bundle}]
\label{def:bp-decompositions}
The parabolic decomposition $w=w^J w_J$ is called a \emph{Billey--Postnikov decomposition} or \emph{BP-decomposition} of $w$ if $\supp(w^J) \cap J \subset D_L(w_J)$.
\end{defin}

\begin{theorem}[Richmond and Slofstra \cite{Richmond-Slofstra-fiber-bundle}]
\label{thm:bp-decomposition-fiber-bundle}
The map $X_w \twoheadrightarrow X^J_{w^J}$ induced by the map $G/B \to G/P_J$ is a bundle projection if and only if $J$ is a BP-decomposition of $w$, and in this case the fiber is isomorphic to $X_{w_J}$. Taking Poincar\'{e} polynomials, we have $L^J(w^J)L(w_J)=L(w)$ in this case.
\end{theorem}

\subsection{Patterns in Weyl groups}
\label{sec:prelim-patterns}
Let $\Phi \subset \Lie_{\R}(T)^{\ast}$ denote the root system for $G$, with positive roots $\Phi^+$ and simple roots $\Delta$. For $w \in W$, the \emph{inversion set} is $\inv(w)\coloneqq \{\alpha \in \Phi^+ \mid w\alpha \in \Phi^-\}$.

A subgroup $W'$ of $W$ generated by reflections is called a \emph{reflection subgroup}, and is itself a Coxeter group with reflections $R'=R \cap W'$. We write $\leq'$ for the intrinsic Bruhat order on $W'$, $\Phi'$ for the root system, and $\inv'$ for inversion sets. 

\begin{prop}[Billey and Braden \cite{Billey-Braden}; Billey and Postnikov \cite{Billey-Postnikov}]
\label{prop:flattening-map}
Let $W' \subset W$ be a reflection subgroup, there is a unique function $\fl: W \to W'$, the \emph{flattening map} satisfying:
\begin{enumerate}
    \item[\normalfont{(1)}] $\fl$ is $W'$-equivariant, and
    \item[\normalfont{(2)}] if $\fl(x) \leq' \fl(wx)$ for some $w \in W'$, then $x \leq wx$.
\end{enumerate}
Furthermore, $\fl$ has the following explicit description: $\fl(w)$ is the unique element $w' \in W'$ with $\inv'(w')=\inv(w) \cap \Phi'$. If $W'=W_J$ is a parabolic subgroup, then $\fl(w)=w_J$.
\end{prop}

\begin{defin}
\label{def:pattern-avoidance}
We say that $w \in W$ \emph{contains the pattern} $w'' \in W''$, if $W$ has some reflection subgroup $W'$, with an isomorphism $W'\xrightarrow{\ph} W''$ as Coxeter systems, such that $\ph(\fl(w))=w''$. Otherwise, $w$ is said to \emph{avoid} $w''$.
\end{defin}

\begin{theorem}[Special case of Theorem 5 of Billey and Braden \cite{Billey-Braden}]
\label{thm:h-increases-in-parabolic}
Let $J \subset S$, then $h(w) \leq h(w_J)$.
\end{theorem}

By Theorem~\ref{thm:h-increases-in-parabolic} and Proposition~\ref{prop:kl-of-inverse}, we also have
\begin{equation}
\label{eq:h-increases-left-parabolic}
h(w) \leq h(\prescript{}{J}{w}).
\end{equation}

Billey and Postnikov proved the following characterization of smooth Schubert varieties, generalizing the work of Lakshmibai--Sandhya \cite{Lakshmibai-Sandhya} in the case $G=\SL_{n}$. We write $W(Z)$ to denote the Weyl group of Type $Z$, where $Z$ is one of the types in the Cartan--Killing classification.

\begin{theorem}[Billey and Postnikov \cite{Billey-Postnikov}]
\label{thm:bp-smoothness}
Let $G$ be simply-laced, then the Schubert variety $X_w \subset G/B$ is smooth if and only if $w$ avoids the following patterns (see Figure~\ref{fig:dynkin-diagrams} for indexing conventions):
\begin{enumerate}
    \item[\normalfont{(1)}] $s_2s_1s_3s_2 \in W(A_3)$,
    \item[\normalfont{(2)}] $s_1s_2s_3s_2s_1 \in W(A_3)$, and
    \item[\normalfont{(3)}] $s_2s_0s_1s_3s_2 \in W(D_4)$.
\end{enumerate}
\end{theorem}

\subsection{Conventions for simply-laced groups}
\subsubsection{$G=\SL_{n}$ (Type $A_{n-1}$)}
We let $B$ be the set of lower triangular matrices in $G$, and $T \subset B$ the diagonal matrices in $G$. We thus have
\begin{align*}
    \Phi(A_{n-1}) &= \{e_j-e_i \mid 1 \leq i \neq j \leq n\} \\
    \Phi^+(A_{n-1}) &= \{e_j-e_i \mid 1 \leq i<j \leq n\} \\
    \Delta(A_{n-1}) &= \{e_{i+1}-e_i \mid 1 \leq i \leq n-1\}.
\end{align*}
Under these conventions, the Weyl group $W(A_{n-1})$ acts on $\Lie_{\R}(T)^* = \R^n / (1,\ldots,1)$ by permutation of the coordinates, yielding an isomorphism $W(A_{n-1}) \cong \S_n$. Letting $\alpha_i \coloneqq e_{i+1}-e_i$, the corresponding simple reflection $s_i$ is identified with the transposition $(i \: i+1) \in \S_n$. It will often be convenient for us to write permutations $w$ in one-line notation as $w(1)\ldots w(n)$. The Dynkin diagram is shown in Figure~\ref{fig:dynkin-diagrams}.

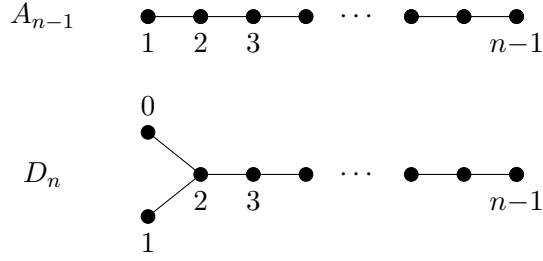
\begin{figure}
\centering
\begin{tikzpicture}[scale=0.7]
\draw(3,0)--(0,0);
\draw(7,0)--(5,0);
\node[draw,shape=circle,fill=black,scale=0.5] at (0,0) {};
\node[draw,shape=circle,fill=black,scale=0.5] at (1,0) {};
\node[draw,shape=circle,fill=black,scale=0.5] at (2,0) {};
\node[draw,shape=circle,fill=black,scale=0.5] at (3,0) {};
\node[draw,shape=circle,fill=black,scale=0.5] at (5,0) {};
\node[draw,shape=circle,fill=black,scale=0.5] at (6,0) {};
\node[draw,shape=circle,fill=black,scale=0.5] at (7,0) {};
\node[draw,shape=circle,fill=black,scale=0.5][label=below: {$1$}] at (0,0) {};
\node[draw,shape=circle,fill=black,scale=0.5][label=below: {$2$}] at (1,0) {};
\node[draw,shape=circle,fill=black,scale=0.5][label=below: {$3$}] at (2,0) {};
\node[draw,shape=circle,fill=black,scale=0.5][label=below: {}] at (3,0) {};
\node[draw,shape=circle,fill=black,scale=0.5][label=below: {}] at (5,0) {};
\node[draw,shape=circle,fill=black,scale=0.5][label=below: {}] at (6,0) {};
\node[draw,shape=circle,fill=black,scale=0.5][label=below: {$n{-}1$}] at (7,0) {};
\node at (4,0) {$\cdots$};
\node at (-2,0) {$A_{n-1}$};

\draw(0,-3.8)--(1,-3)--(0,-2.2);
\draw(1,-3)--(3,-3);
\draw(5,-3)--(7,-3);
\node at (4,-3) {$\cdots$};
\node[draw,shape=circle,fill=black,scale=0.5][label=above: {$0$}] at (0,-2.2) {};
\node[draw,shape=circle,fill=black,scale=0.5][label=below: {$1$}] at (0,-3.8) {};
\node[draw,shape=circle,fill=black,scale=0.5][label=below: {$2$}] at (1,-3) {};
\node[draw,shape=circle,fill=black,scale=0.5][label=below: {$3$}] at (2,-3) {};
\node[draw,shape=circle,fill=black,scale=0.5][label=below: {}] at (3,-3) {};
\node[draw,shape=circle,fill=black,scale=0.5][label=below: {}] at (5,-3) {};
\node[draw,shape=circle,fill=black,scale=0.5][label=below: {}] at (6,-3) {};
\node[draw,shape=circle,fill=black,scale=0.5] at (7,-3) {};
\node[draw,shape=circle,fill=black,scale=0.5][label=below: {$n{-}1$}] at (7,-3) {};
\node at (-2,-3) {$D_n$};
\end{tikzpicture}
    \caption{The Dynkin diagrams for Types $A_{n-1}$ and $D_n$, with the labelling of the nodes that we use throughout the paper.}
    \label{fig:dynkin-diagrams}
\end{figure}

\subsubsection{$G=\SO_{2n}$ (Type $D_n$)}
We let $B$ be the set of lower triangular matrices in $G$, and $T \subset B$ the diagonal matrices in $G$. We thus have
\begin{align*}
    \Phi(D_n) &= \{e_j \pm e_i \mid 1 \leq i \neq j \leq n\} \\
    \Phi^+(D_n) &= \{e_j \pm e_i \mid 1 \leq i<j \leq n\} \\
    \Delta(D_n) &= \{e_2+e_1\} \cup \{e_{i+1}-e_i \mid 1 \leq i \leq n-1\}.
\end{align*}
Under these conventions, the Weyl group $W(D_n)$ acts on $\Lie_{\R}(T)^* = \R^n$ by permuting coordinates and negating pairs of coordinates. This identifies $W(D_n)$ with the subgroup of the symmetric group on $\{-n, \ldots, -1, 1, \ldots, n\}$ satisfying $w(i)=-w(-i)$ for all $i$, and such that \[
|\{w(1),\ldots,w(n)\}\cap \{-n,\ldots,-1\}|\]
is even. We write $\D_n$ for this realization of $W(D_n)$. Such a permutation can be uniquely specified by its \emph{window notation} $[w(1)\ldots w(n)]$.

Write $\delta_0=e_2+e_1$ and $\delta_i=e_{i+1}-e_i$, $i=1,2,\ldots,n-1$ for the simple roots. It will often be convenient for us to write $\bar{i}$ for $-i$, and we use these interchangeably. We also make the convention that $e_{\bar i}=e_{-i}\coloneqq -e_i$ for $i>0$. We have simple reflections $s_0=(1 \: \bar{2}) (\bar{1} \: 2)$ and $s_i=(i \: i+1)(\bar{i} \: \overline{i{+}1})$ for $i=1,\ldots,n-1$. The Dynkin diagram is shown in Figure~\ref{fig:dynkin-diagrams}.

\subsection{Reflection subgroups and diagram automorphisms}

See Figure~\ref{fig:dynkin-diagrams} for our labeling of the Dynkin diagrams. The following is clear:

\begin{prop}
\label{prop:diagram-automorphisms}
The diagram of the Type $A_{n-1}$ has an automorphism $\e_A$ sending $\alpha_i \mapsto \alpha_{n-i}$ for $i=1,\ldots,n-1$, and the diagram of Type $D_n$ has an automorphism $\e_D$ interchanging $\delta_0\leftrightarrow \delta_1$. On the Weyl groups, this induces:
\[w(1)\ldots w(n) \xmapsto{\e_A} (n+1-w(n)) \ldots (n+1-w(1)),\]
\[w \xmapsto{\e_D} (1 \: \bar{1}) \cdot w \cdot (1 \: \bar{1}). \]
\end{prop}

It is clear from Definition~\ref{def:kl-polys} that we have:

\begin{prop}
\label{prop:auto-preserves-h}
If $w \in W(A_{n-1})$ then $h(w)=h(\e_A(w))$, and if $w \in W(D_n)$ then $h(w)=h(\e_D(w))$.
\end{prop}

\subsubsection{Reflection subgroups}

By Theorem~\ref{thm:bp-smoothness}, we will be concerned with reflection subgroups isomorphic to $W(A_3)$ and $W(D_4)$ inside $W(A_{n-1})$ and $W(D_n)$.

The following classification follows from, e.g., Haenni \cite{Haenni}.

\begin{prop}
\label{prop:reflection-subgroups}
Reflection subgroups isomorphic to $W(A_3)$ and $W(D_4)$ inside $W(A_{n-1})$ and $W(D_n)$ are characterized as follows:
\begin{itemize}
    \item[(a)] No reflection subgroup $W' \subset W(A_{n-1})$ is isomorphic to $W(D_4)$,
    \item[(b)] Reflection subgroups $W' \cong W(A_3)$ inside $W(A_{n-1})$ are conjugate to the parabolic subgroup $W(A_{n-1})_{\{1,2,3\}}$,
    \item[(c)] Reflection subgroups $W' \cong W(D_4)$ inside $W(D_n)$ are conjugate to the parabolic subgroup $W(D_n)_{\{0,1,2,3\}}$.
    \item[(d)] Reflection subgroups $W' \cong W(A_3)$ inside  $W(D_n)$ come in two classes: those related to $W(D_n)_{\{1,2,3\}}$ by conjugacy and $\e_D$ (Class I), and those conjugate to $W(D_n)_{\{0,1,2\}}$ (Class II).
\end{itemize}
\end{prop}

\subsubsection{One line notation and patterns}
\label{sec:pattern-conventions}
We will be interested in occurrences of the patterns from Theorem~\ref{thm:bp-smoothness} in elements $w \in W(A_{n-1})$ or $W(D_n)$. For $w \in W(D_n)$, it will sometimes be useful for us to distinguish between Class I and II patterns (see Proposition~\ref{prop:reflection-subgroups}). Realizing these Weyl groups as $\S_n$ and $\D_n$, respectively, we have the following interpretations of pattern containment (summarized in Figure~\ref{fig:smooth-patterns}). This approach to pattern containment is in some sense a hybrid between the approaches of Billey \cite{Billey-signed-patterns} using signed patterns and of Billey, Braden, and Postnikov \cite{Billey-Postnikov, Billey-Braden} using patterns in the sense of Definition~\ref{def:pattern-avoidance}. Our distinction between Class I and II patterns is seemingly novel and reflects the disparate effects that occurrences of these patterns can have on $h(w)$.

\begin{defin} \text{}
\label{def:signed-pattern}
\begin{itemize}
    \item[(i)] For $p$ a signed permutation of $[k]$, we say $w \in \D_n$ contains $p$ at positions $1 \leq i_1 < \cdots < i_k \leq n$ if $\sgn(w(i_j))=\sgn(p(j))$ for $j=1,\ldots,k$ and $|w(i_1)|,\ldots,|w(i_k)|$ are in the same relative order as $|p(1)|,\ldots,|p(k)|$.
    \item[(ii)] For $p \in \S_k$, we say $w \in \S_n$ contains $p$ at positions $1 \leq i_1 < \cdots < i_k \leq n$ if $w(i_1),\ldots,w(i_k)$ have the same relative order as $p(1),\ldots,p(k)$. We say $u \in \D_n$ contains $p$ at positions $i_1 < \cdots <i_k$, where each $i_j \in \pm [n]$ if $u(i_1),\ldots,u(i_k)$ have the same relative order as $p(1),\ldots,p(k)$ and $|i_1|,\ldots,|i_k|$ are distinct.
\end{itemize}
In each case, we say that the \emph{values} of the occurrence are $w(i_1),\ldots,w(i_k)$.
\end{defin}

\begin{remark}
Whenever we refer to an occurrence of $3412$ or $4231$ in $w \in \D_n$, we always mean an occurrence in the sense of Definition~\ref{def:signed-pattern}(ii).
\end{remark}

The following is a translation of Theorem~\ref{thm:bp-smoothness} in light of our conventions for patterns. We use $\pm$ to indicate pairs of patterns differing in the sign of the first position; for example, $\pm 1 2 \bar{3}$ refers to the two patterns $\bar{1} 2 \bar{3}$ and $1 2 \bar{3}$.

\begin{prop}
\label{prop:hybrid-patterns}
Let $G$ be simply-laced; then $X_w \subset G/B$ is smooth if and only if $w$ avoids the patterns $3412, \pm 1 2 \bar{3}, 4231, \pm 1 \bar{3} \bar{2},$ and $\pm 1 4\bar{3}2$ (see Figure~\ref{fig:smooth-patterns}).
\end{prop}

\begin{figure}
    \centering
    \begin{tabular}{|c|c|c|c|}
    \hline 
    Type & Class & Pattern & One-line \\  \hline
    $A_3$ & I & $s_2s_1s_3s_2$ & $3412$ \\ \hline
    $A_3$ & II & $s_2s_1s_3s_2$ & $\pm 1 2 \bar{3}$ \\ \hline
    $A_3$ & I & $s_1s_2s_3s_2s_1$ & $4231$ \\ \hline
    $A_3$ & II & $s_1s_2s_3s_2s_1$ & $\pm 1 \bar{3} \bar{2}$ \\ \hline
    $D_4$ &  & $s_2s_0s_1s_3s_2$ & $\pm 1 4 \bar{3} 2$ \\ \hline
    \end{tabular}
    \caption{The patterns from Theorem~\ref{thm:bp-smoothness} with their one-line notations, divided according to type and class as discussed in Section~\ref{sec:pattern-conventions}}
    \label{fig:smooth-patterns}
\end{figure}

The following statistic on occurrences of the pattern $3412$ will be of special importance for us (see Theorem~\ref{thm:type-A-formula}).

\begin{defin}[See \cite{Cortez-2, Woo}]
\label{def:height}
We say an occurrence of $3412$ in $w \in \S_n$ or $\D_n$ at positions $a<b<c<d$ has \emph{height} equal to $w(a)-w(d)$. We let $\mh(w)$ denote the minimum height over all occurrences of $3412$ in $w$.
\end{defin}

\section{Upper bounds on $h(w)$}
\label{sec:upper-bounds}
\subsection{Proof strategy}
We will identify certain patterns $p$ (among those from Proposition~\ref{prop:hybrid-patterns}) such that if $w$ contains $p$, then $h(w)$ can be computed using Theorem~\ref{thm:bjorner-ekedahl} and an analysis of the Bruhat covers of $w$. Then, for $w$ avoiding these patterns and containing others, we will---by a combination of parabolic reduction (Theorem~\ref{thm:h-increases-in-parabolic}), inversion (Proposition~\ref{prop:kl-of-inverse}), and diagram automorphisms (Proposition~\ref{prop:auto-preserves-h})---obtain a bound $h(w) \leq h(u)$ for $u$ in some special family $\mathcal{S}$. Finally, we will show that elements $u \in \mathcal{S}$ have distinguished BP-decompositions such that the base and fiber in the bundle (Theorem~\ref{thm:bp-decomposition-fiber-bundle}) with total space $X_u$ can be understood, allowing for the computation of $h(u)$. For convenience, in the remainder of the paper we will refer primarily to the elements $w \in W$ rather than the Schubert varieties $X_w$ that they index, although each of these steps has a geometric basis. We say $w$ is \emph{smooth} (resp. \emph{singular}) if $X_w$ is smooth (resp. singular).

\subsection{Relations among Bruhat covers}
Write $r_{\beta}$ for the reflection corresponding to the root $\beta \in \Phi^+$.

\begin{lemma}\label{lem:relation-corank1}
Let $wr_{\beta_1}, \ldots, wr_{\beta_k}$ be the elements covered by $w$ in Bruhat order. If $\beta_1,\ldots, \beta_k$ are linearly dependent, then $h(w)=1$.
\end{lemma}
\begin{proof}
Let $\gamma_1,\ldots,\gamma_m \in \Delta$ be the simple roots whose corresponding simple reflections $r_{\gamma_i}$ are $\leq w$. By results of Dyer \cite{Dyer}, we have
\[\spn_{\R}(\gamma_1,\ldots,\gamma_m)=\spn_{\R}(\beta_1,\ldots,\beta_k).\]
The $\{\gamma_i\}_i$ are linearly independent, so the result follows by Theorem~\ref{thm:bjorner-ekedahl}.
\end{proof}

\begin{lemma}\label{lem:split-inversion-once}
Let $w\in W$ be simply-laced, and $\beta\in\inv(w)$. If $wr_{\beta}$ is not covered by $w$, then there exist $\beta_1,\beta_2\in\inv(w)$ such that $\beta_1+\beta_2=\beta$.
\end{lemma}
\begin{proof}
Given $\beta\in\Phi^+$, divide the positive roots into $\Phi^+=\{\beta\}\sqcup A\sqcup B$ where \[A=\{\alpha\in\Phi^+\:|\: r_{\beta}\alpha\in\Phi^+\},\quad B=\{\alpha\in\Phi^+\:|\:r_{\beta}\alpha\in\Phi^-,\alpha\neq\beta\}.\]
Note that $\beta\in\inv(w)$ and $\beta\notin\inv(wr_{\beta})$. We have $\alpha\in A\cap\inv(w)$ if and only if $r_{\beta}\alpha\in A\cap\inv(wr_{\beta})$, this determines a bijection between $A\cap\inv(w)$ and $A\cap\inv(wr_{\beta})$. Now for $\alpha\in B$, we have $r_{\beta}\alpha=\alpha-\frac{2\langle\alpha,\beta\rangle}{\langle\beta,\beta\rangle}\beta\in\Phi^-$ so we must have $r_{\beta}\alpha=\alpha-\beta$ with $\beta-\alpha\in\Phi^+$. We can pair each root $\alpha$ in $B$ with $\beta-\alpha=-r_{\beta}\alpha$. As $\beta\in\inv(w)$, at least one of $\alpha,\beta-\alpha$ is in $\inv(w)$. Moreover, $\alpha\in\inv(w)$ if and only if $\beta-\alpha\notin\inv(wr_{\beta})$. If exactly one of $\alpha,\beta-\alpha$ is in $\inv(w)$ for all such pairs in $B$, then we see that $|\inv(w)|-|\inv(wr_{\beta})|=1$, contradicting the assumption that $w$ does not cover $wr_{\beta}$. As a result, there must be some $\alpha,\beta-\alpha\in\inv(w)$. 
\end{proof}

A repeated application of Lemma~\ref{lem:split-inversion-once} gives the following:
\begin{lemma}\label{lem:split-inversion-all}
Let $w\in W$ be simply-laced and $\beta\in\inv(w)$. Then there exist $\beta_1,\ldots,\beta_k\in\inv(w)$ such that $\beta=\beta_1+\cdots+\beta_k$ and $w\gtrdot wr_{\beta_i}$ for $i=1,\ldots,k$. Note that the choice of $\beta_1,\ldots,\beta_k$ is not unique.
\end{lemma}

If $w\gtrdot wr_{\beta}$ we sometimes say that $\beta$ is \emph{a label below $w$}. 

\begin{prop}\label{prop:h-equals-1}
Let $w \in \S_n$ or $\D_n$; we have $h(w)=1$ if $w$ contains:
\begin{itemize}
    \item[(i)] $4231$ and $w \in \S_n$,
    \item[(ii)] $\pm 1 2 \bar{3}$,
    \item[(iii)] $\pm 1 4 \bar{3} 2$, or
    \item[(iv)] $3412$ of height one.
\end{itemize}
\end{prop}

\begin{proof}
The strategies for all cases are exactly the same. A bad pattern implies a relation $\tau_1+\tau_2=\tau_3+\tau_4$ for $\tau_1,\tau_2,\tau_3,\tau_4\in\inv(w)$. Then we use Lemma~\ref{lem:split-inversion-all} to write this equation as a relation for the labels of covers below $w$. The key step is proving that this relation is nontrivial. Finally, we can conclude with Lemma~\ref{lem:relation-corank1}. This strategy will also be used in Proposition~\ref{prop:type-D-4231-upper-bound}.

Case (i). Suppose that $w$ contains $4231$ at indices $a<b<c<d$. We then have a relation $(e_d-e_b)+(e_b-e_a)=(e_d-e_c)+(e_c-e_a)$ with $e_b-e_a,e_d-e_b,e_c-e_a,e_d-e_c\in\inv(w)$. By Lemma~\ref{lem:split-inversion-all}, we can further split the left hand side into $e_{p_k}-e_{p_{k-1}},\ldots,e_{p_2}-e_{p_1},e_{p_1}-e_{p_0}$ and the right hand side into $e_{q_m}-e_{q_{m-1}},\ldots,e_{q_{1}}-e_{q_0}$ where $p_0=q_0=a$, $p_k=q_m=d$, and each $wr_{e_{p_{i+1}}-e_{p_{i}}}$, $wr_{e_{q_{j+1}}-e_{q_{j}}}$ is covered by $w$. We have $p_i=b$, $q_j=c$ for some $i,j$ and $e_c-e_b$ is not an inversion of $w$, meaning that $c$ cannot appear as an endpoint $p_{i'}$ on the left hand side. But $c$ is an endpoint $q_j$ on the right hand side. Thus, after canceling some common entries, the equation \[(e_{p_k}-e_{p_{k-1}})+\cdots+(e_{p_1}-e_{p_0})=(e_{q_{m}}-e_{q_{m-1}})+\cdots+(e_{q_1}-e_{q_0})\]
is still nontrivial. We then apply Lemma~\ref{lem:relation-corank1} to conclude that $h(w)=1$.

Case (ii). Suppose that $w$ contains $\pm12\bar3$ at indices $0<a<b<c$, then\[(e_c-e_b)+(e_c+e_b)=(e_c-e_a)+(e_c+e_a).\]
It is more convenient to view this equation as \[(e_c-e_b)+(e_b-e_{\bar c})=(e_c-e_a)+(e_a-e_{\bar c})\]
where we adopt the convention that $e_{\bar i}=-e_i$. Further split these four terms using Lemma~\ref{lem:split-inversion-all} into covers below $w$. As before, since $b$ is an endpoint of the left hand side and $w(b)>w(a), w(\bar a)$, we know that neither $a$ nor $\bar a$ can appear on the left hand side. As a result, we have obtained a nontrivial relation among the positive roots that label cover relations below $w$, and Lemma~\ref{lem:relation-corank1} implies that $h(w)=1$.

Case (iii). Suppose that $w$ contains $\pm14\bar32$ at indices $0<a<b<c<d$, and consider the following equation of positive roots in $\inv(w)$: \[(e_c-e_b)+(e_d+e_c)=(e_d-e_b)+(e_c-e_a)+(e_c+e_a).\]
As we split these roots according to Lemma~\ref{lem:split-inversion-all}, $e_a$ does not appear on the left hand side because $\pm a$ do not belong to the interval $(b,c)$ and $w(\pm a)$ do not belong to the interval $(w(c),w(\bar d))$. Thus, we obtain a nontrivial relation among labels below $w$, so $h(w)=1$.

Case (iv). Suppose that $w$ contains an occurrence of $3412$ of height one at indices $a<b<c<d$. Then \[(e_d-e_a)+(e_c-e_b)=(e_d-e_b)+(e_c-e_a).\]
Since this occurrence has height one, $w$ covers $wr_{e_d-e_a}$. However, $e_d-e_a$ cannot possibly appear on the right hand side after applying the splitting in Lemma~\ref{lem:split-inversion-all}, as both $e_d-e_b$ and $e_c-e_a$ are smaller than $e_d-e_a$ in the root order (that is, they are equal to $e_d-e_a$ minus a positive linear combination of the positive roots). Thus a nontrivial relation is produced and $h(w)=1$ by Lemma~\ref{lem:relation-corank1}. 
\end{proof}
\begin{ex}
We use a non-example to demonstrate a key step in the proof of Proposition~\ref{prop:h-equals-1}. Suppose that $w=45312$, which contains an occurrence of $3412$ of height $2$. The $3412$ pattern allows us to write \[(e_4-e_1)+(e_5-e_2)=(e_5-e_1)+(e_4-e_2).\]
As $w(3)=3$, none of these four roots is a label below $w$. After splitting as in Lemma~\ref{lem:split-inversion-all}, the equation is reduced to the trivial relation \[(e_4-e_3)+(e_3-e_1)+(e_5-e_3)+(e_3-e_2)=(e_5-e_3)+(e_3-e_1)+(e_4-e_3)+(e_3-e_2).\] 
\end{ex}

\subsection{Proof of Theorem~\ref{thm:simply-laced-upper-bound} in Type $A$}\label{sub:typeA-upper-bound}
In this section we obtain an upper bound on $h(w)$ for $w \in \S_n$ in terms of $\mh(w)$; this establishes Theorem~\ref{thm:simply-laced-upper-bound} for $W=\S_n$ as well as one direction of Theorem~\ref{thm:type-A-formula}.

\begin{defin}\label{def:palindromic-to-k}
For a polynomial $f\in\Z_{\geq0}[q]$ of degree $d$, we say that $f$ is \emph{top-heavy} if $[q^i]f\leq[q^{d-i}]f$ for all $0\leq i\leq d/2$. We define \[h(f):=\min\{0\leq i\leq d/2\:|\: [q^i]f<[q^{d-i}]f\}.\] 
If $f$ is palindromic, we make the convention that $h(f)=+\infty$. 
\end{defin}

The following lemma is a simple exercise, for which we omit the proof.
\begin{lemma}\label{lem:min-of-h}
If polynomials $f_1,f_2 \in\Z_{\geq0}[q]$ are top-heavy, then we have $h(f_1f_2)=\min(h(f_1),h(f_2))$.
\end{lemma}

\noindent Note that the product of top-heavy polynomials need not be top-heavy.

We will apply Lemma~\ref{lem:min-of-h} to $L^{J}(w^J)$ and $L(w_J)$, with comparison to $L(w)$. 

\begin{lemma}\label{lem:typeA-extremal-case}
For $n\geq4$, consider $w\in \S_n$ where $w(1)=n-1$, $w(2)=n$, $w(n-1)=1$, $w(n)=2$ and $w(i)=n-i+1$ for $3\leq i\leq n-2$. Then $h(w)=n-3$. 
\end{lemma}
\begin{proof}
Let $J=\{2,3,\ldots,n-2\}$ so that $w_J=w_0(J)$. The parabolic decomposition $w=w^Jw_J$ is a Billey--Postnikov decomposition. Moreover, $L(w_J)=L(w_0(J))$ is palindromic, since $X_{w_0(J)}$ is a flag variety and therefore smooth. Every $u\in W^J$ satisfies $u(2)<u(3)<\cdots<u(n{-}1)$ so by counting inversions with $u(1)$ and $u(n)$, we deduce that $\ell(u)=(u(1)-1)+(n-u(n))-\boldsymbol{1}_{u(1)>u(n)}$, where $\boldsymbol{1}_{u(1)>u(n)}$ equals $1$ if $u(1)>u(n)$ and equals $0$ if $u(1)<u(n)$. Elements $u\in[e,w^J]^J$ are characterized by $u(1)\leq n-1$ and $u(n)\geq2$ with $u(2)<\cdots<u(n{-}1)$. With this criterion and the length formula, we are now able to directly count the rank sizes of $[e,w^J]^J$, starting at rank $0$ and ending at rank $2n-5$ to be \[1,2,3,\ldots,n-4,n-3,n-2,n-1,n-3,n-4,\ldots,2,1.\]
Thus, $h(L^J(w^J))=n-3$ and by Lemma~\ref{lem:min-of-h}, \[h(w)=h(L(w))=\min(h(L^J(w^J)),h(L(w_J)))=\min(n-3,\infty)=n-3.\]
\end{proof}

Recall that for an occurrence of a $3412$ in $w$ at indices $a<b<c<d$ with $w(c)<w(d)<w(a)<w(b)$, its \emph{height} is $w(a)-w(d)$. We define the \emph{content} of such an occurrence to be \[1+\left|\{i\:|\: b<i<c, w(d)<w(i)<w(a)\}\right|.\]
Let $\mcontent(w)$ be the minimum content of a $3412$ pattern in $w$.
\begin{lemma}
For $w\in\S_n$ that contains $3412$, $\mh(w)=\mcontent(w)$.
\end{lemma}
\begin{proof}
Since the content of each occurrence is bounded above by the height, we have $\mcontent(w)\leq \mh(w)$. Now suppose that $\mcontent(w)=k$ and choose an occurrence of $3412$ of content $k$ in $w$ at indices $a<b<c<d$, $w(c)<w(d)<w(a)<w(b)$ such that $w(a)-w(d)$ is minimized. For any $i$ such that $w(d)<w(i)<w(a)$, if $i<b$ the occurrence of $3412$ at indices $i<b<c<d$ (of content at most $k$) has a smaller value of $w(i)-w(d)$, contradicting the minimality of $w(a)-w(d)$. Thus, for any $i$ such that $w(d)<w(i)<w(a)$, we have $i>b$, and, likewise, $i<c$. As a result, for this particular $3412$, its height equals its content $k$. So $\mh(w)\leq k=\mcontent(w)$ as desired. 
\end{proof}
One advantage of working with content instead of height, is that we evidently have $\mcontent(w)=\mcontent(w^{-1})$.
\begin{lemma}\label{lem:typeA-3412-upper}
Suppose that $w\in \S_n$ avoids $4231$ and contains $3412$. Then $h(w)\leq \mh(w)$. 
\end{lemma}
\begin{proof}
We use induction on $n$. The statement is true when $n=4$, where $h(3412)=\mh(3412)=1$. For a general $n$ and $w\in\S_n$, let $k=\mh(w)=\mcontent(w)$. For $J=\{2,3,\ldots,n-1\}$, if $w_J\in\S_{n-1}$ has $\mcontent(w_J)=k$, then we are done by the induction hypothesis and Theorem~\ref{thm:h-increases-in-parabolic} which says $h(w)\leq h(w_J)\leq \mcontent(w_J)=k$. We can thus assume without loss of generality that the index $1$ appears in all $3412$'s of $w$ with content $k$. Similarly by considering $J=\{1,2,\ldots,n-2\}$, we can also assume that the index $n$ appears in all $3412$'s of $w$ with content $k$. As $h(w)=h(w^{-1})$, with the same argument on $w^{-1}$, we can reduce to the case that $w$ contains a unique $3412$ of content $k$ on indices $1<w^{-1}(n)<w^{-1}(1)<n$ (see Figure~\ref{fig:3412-upper-bound}).
\begin{figure}[h!]
\centering
\begin{tikzpicture}[scale=0.8]
\draw[step=1.0] (0,0) grid (3,3);
\node at (0,2) {$\bullet$};
\node at (1,3) {$\bullet$};
\node at (2,0) {$\bullet$};
\node at (3,1) {$\bullet$};
\node at (0.5,0.5) {$A$};
\node at (1.5,1.5) {$B$};
\node at (2.5,2.5) {$C$};
\node at (0.5,1.5) {$\emptyset$};
\node at (0.5,2.5) {$\emptyset$};
\node at (1.5,0.5) {$\emptyset$};
\node at (1.5,2.5) {$\emptyset$};
\node at (2.5,0.5) {$\emptyset$};
\node at (2.5,1.5) {$\emptyset$};
\end{tikzpicture}
\caption{The permutation diagram for $w$ with an occurrence of $3412$ on the boundary. Throughout the paper, we draw permutation diagrams by putting $\bullet$'s at Cartesian coordinates $(i,w(i))$.}
\label{fig:3412-upper-bound}
\end{figure}
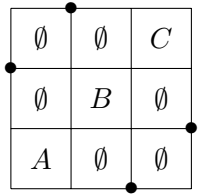
As we assume that $w_J$ does not contain a $3412$ of content $k$, there does not exist $i$ such that $1<i<w^{-1}(n)$ with $w(i)>w(n)$. By symmetry, we know six of the regions in Figure~\ref{fig:3412-upper-bound} are empty as shown, and label the other three regions as $A,B,C$. By definition, $|B|=k-1$. If $k=1$, then $h(w)=1$ by Proposition~\ref{prop:h-equals-1}. If $k>1$, $B$ is not empty; since $w$ avoids $4231$, $A$ and $C$ must be empty. Finally, in this case the entries of $B$ must be decreasing, since $w$ avoids $4231$. Thus $w$ is exactly the permutation in Lemma~\ref{lem:typeA-extremal-case}, which gives $h(w)=n-3=k$ as desired. 
\end{proof}

\subsection{Extension to Type $D$}

\begin{prop}
\label{prop:type-D-4231-upper-bound}
If $w \in \D_n$ contains $4231$, then $h(w) \leq 2$.
\end{prop}
\begin{proof}
We will adapt the strategy for Proposition~\ref{prop:h-equals-1} to show that for most occurrences of $4231$, we in fact have $h(w)=1$. For the few remaining cases, we apply a different argument. 

Assume that $w$ contains $4231$ at indices $a<b<c<d$. Recall that, by our conventions, this means in particular that $|a|,|b|,|c|,|d|$ are distinct, as are $|w(a)|,|w(b)|,|w(c)|,|w(d)|$. We have the following equality among roots in $\inv(w):$ 
\begin{equation}\label{eq:root-h-2}
(e_d-e_b)+(e_b-e_a)=(e_d-e_c)+(e_c-e_a),
\end{equation}
where $e_{\bar i}=-e_i$. In this argument, we write all the roots in the form $e_{i}-e_{j}$, where the subscript can be positive or negative. By Lemma~\ref{lem:split-inversion-all}, we can split each of these four roots into a sum of labels below $w$. If we end up with a nontrivial relation, we can conclude that $h(w)=1$ by Lemma~\ref{lem:relation-corank1}. We now analyze when Equation~\eqref{eq:root-h-2} becomes trivial and deal with these situations. As $w(c)>w(b)>w(d)$, $c$ cannot appear as an endpoint of a root after we split $e_d-e_b=(e_d-e_i)+\cdots+(e_j-e_a)$, and as $c>b>a$, $c$ cannot appear in the expansion of $e_b-e_a$ either. Thus, we need $\bar c$ to appear on the left hand side of Equation~\eqref{eq:root-h-2}, noticing that $e_i-e_j=e_{\bar j}-e_{\bar i}$. Likewise, we need $\bar b$ to appear on the right hand side. There are two cases.

Case 1: $\bar c$ appears in $e_b-e_a$. This implies that $a<\bar c<b$ and $w(b)<w(\bar c)<w(a)$. As $a<\bar c<b<c<d$, we have $c>0$ and $c>a,-d,|b|$. Also, $w(a)>w(c)>w(b)>w(d)$, which gives \[w(a)>|w(c)|>0>-|w(c)|>w(b)>w(d).\] Simultaneously, $\bar b$ needs to appear on the right hand side of Equation~\eqref{eq:root-h-2}. As $\bar b<c<d$, we know $\bar b$ does not appear in the expansion of $e_d-e_c$, so it must appear in $e_c-e_a$, which implies $w(a)>w(\bar b)>w(c)$ so $w(a),w(\bar d)>w(\bar b)>|w(c)|$. The relative orderings of $|a|,|b|,|c|,|d|$ and of $|w(a)|,|w(b)|,|w(c)|,|w(d)|$ let us conclude that $w$ contains one of the following $8$ patterns: \[\pm2,\pm1,\bar 3,\bar4\text{  and 
 }\pm2,\pm1,\bar4,\bar3.\]

Case 2: $\bar c$ appears in $e_d-e_b$. This implies that $b<\bar c<d$ so $b<-|c|<0<|c|<d$, and that $w(c)>w(b)>w(\bar c)>w(d)$ so $w(c)>0$. As $\bar b>c>a$, we know $\bar b$ cannot appear in $e_c-e_a$ and must appear in $e_d-e_c$, which implies $d>\bar b>c>b$. Together, we have as above that $\bar a, d>\bar b,c>0$ and $w(a),w(\bar d)>w(c),|w(b)|>0$. We also conclude that $w$ contains one of $\pm2,\pm1,\bar3,\bar4$ and $\pm2,\pm1,\bar4,\bar3$.

We establish the following claim first before diving into these $8$ patterns.
\begin{claim}\label{claim:greater-than-rank2}
Let $w\in\D_n$ satisfy $w(n)=\bar n$. Then $w\geq u$ for any $u\in\D_n$ of length $2$.
\end{claim}
\begin{proof}
Let $J=\{0,1,\ldots,n-2\}\subset S$, then \[w^J=\pm123\cdots(n{-}1)\bar n=s_{n-1}s_{n-2}\cdots s_3s_2s_0s_1s_2s_3\cdots s_{n-1},\] which contains as a subword every group element $u=s_is_j$ of length $2$. So $w\geq w^J\geq u$ by Theorem~\ref{thm:subword-bruhat}.
\end{proof}
We now use induction on $n$ to show that if $w\in\D_n$ contains $4231$, then $h(w)\leq 2$. The base case $n=3$ holds vacuously. For a general $n$, by the induction hypothesis, we can assume that $w_J$ and $_Jw$ do not contain $4231$ for $J\subsetneq S$. We may also assume that $w$ avoids the patterns in Proposition~\ref{prop:h-equals-1}, otherwise $h(w)=1$ and we are done. In particular, one of $(n,w(n)),(\bar n,w(\bar n))$ and one of $(w^{-1}(n),n),(w^{-1}(\bar n),\bar n)$ must be used in the patterns $\pm2,\pm1,\bar3,\bar4$ and $\pm2,\pm1,\bar4,\bar3$ of interest.

Case 1: $w$ contains $\pm2,\pm1,\bar4,\bar3$ at indices $0<a<b<c<d$. We can deal with these four patterns using the same argument, but for simplicity we will only consider the most representative case of $21\bar4\bar3$. We apply the diagram analysis as shown in Figure~\ref{fig:21bar4bar3}, where the regions marked $\emptyset$ are empty because $w_J,_Jw$ avoid $4231$ for $J=\{0,1,\ldots,n-2\}$. 
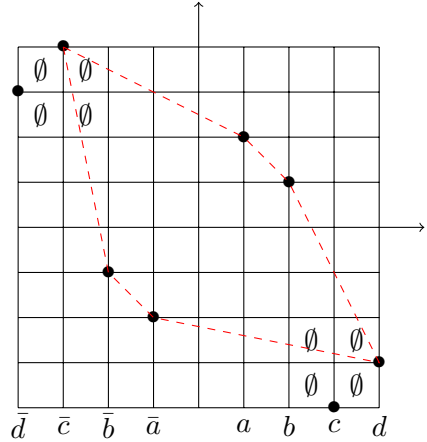
\begin{figure}[h!]
\centering
\begin{tikzpicture}[scale=0.6]
\draw[step=1.0] (-4,-4) grid (4,4);
\node at (1,2) {$\bullet$};
\node at (2,1) {$\bullet$};
\node at (3,-4) {$\bullet$};
\node at (4,-3) {$\bullet$};
\node at (-1,-2) {$\bullet$};
\node at (-2,-1) {$\bullet$};
\node at (-3,4) {$\bullet$};
\node at (-4,3) {$\bullet$};
\node at (1,-4.4) {$a$};
\node at (2,-4.4) {$b$};
\node at (3,-4.4) {$c$};
\node at (4,-4.4) {$d$};
\node at (-1,-4.4) {$\bar a$};
\node at (-2,-4.4) {$\bar b$};
\node at (-3,-4.4) {$\bar c$};
\node at (-4,-4.4) {$\bar d$};
\draw[->](4,0)--(5,0);
\draw[->](0,4)--(0,5);
\node at (2.5,-2.5) {$\emptyset$};
\node at (2.5,-3.5) {$\emptyset$};
\node at (3.5,-2.5) {$\emptyset$};
\node at (3.5,-3.5) {$\emptyset$};
\node at (-2.5,2.5) {$\emptyset$};
\node at (-2.5,3.5) {$\emptyset$};
\node at (-3.5,2.5) {$\emptyset$};
\node at (-3.5,3.5) {$\emptyset$};
\draw[dashed, red](-3,4)--(-2,-1)--(-1,-2)--(4,-3)--(2,1)--(1,2)--(-3,4);
\end{tikzpicture}
\caption{Diagram analysis for $21\bar4\bar3$}
\label{fig:21bar4bar3}
\end{figure}
Let us now consider \[(e_d-e_{\bar a})+(e_{\bar a}-e_{\bar b})+(e_{\bar b}-e_{\bar c})=(e_d-e_{b})+(e_{b}-e_a)+(e_a-e_{\bar c})\] that corresponds to the red dashed line in Figure~\ref{fig:21bar4bar3},
and split these roots according to Lemma~\ref{lem:split-inversion-all}. On the left hand side, there is a term $e_k-e_{\bar c}$ for $\bar c<k\leq\bar b$ where $w\gtrdot wr_{e_k-e_{\bar c}}$. However, on the right hand side, $c$ or $\bar c$ can only appear in the form of $e_l-e_{\bar c}$ for $l>\bar b$. As a result, we obtain a nontrivial linear relation among the positive roots that label covers below $w$. Lemma~\ref{lem:relation-corank1} implies that $h(w)=1$.

We do note that containing $\pm2,\pm1,\bar4,\bar3$ does not imply $h(w)=1$. The condition of the ``minimality" of such an occurrence does.

Case 2: $w$ contains $\pm2,\pm1,\bar3,\bar4$ at indices $0<a<b<c<d$. Again, the four patterns here can be dealt with the same argument so for simplicity, we consider the most representative case $21\bar3\bar4$. Consider Figure~\ref{fig:21bar3bar4}; the regions marked $\emptyset$ are empty because:
\begin{itemize}
\item $\emptyset_1:$ $w_J,_Jw$ need to avoid $4231$ for $J$ of type $A_{n-1}$ and $D_{n-1}$.
\item $\emptyset_2$: $w$ avoids $\pm12\bar3$.
\end{itemize}
\begin{figure}[h!]
\centering
\begin{tikzpicture}[scale=0.6]
\draw[step=1.0] (-4,-4) grid (4,4);
\node at (1,2) {$\bullet$};
\node at (2,1) {$\bullet$};
\node at (3,-3) {$\bullet$};
\node at (4,-4) {$\bullet$};
\node at (-1,-2) {$\bullet$};
\node at (-2,-1) {$\bullet$};
\node at (-3,3) {$\bullet$};
\node at (-4,4) {$\bullet$};
\draw[->](4,0)--(5,0);
\draw[->](0,4)--(0,5);
\node at (1,-4.4) {$a$};
\node at (2,-4.4) {$b$};
\node at (3,-4.4) {$c$};
\node at (4,-4.4) {$d$};
\node at (-1,-4.4) {$\bar a$};
\node at (-2,-4.4) {$\bar b$};
\node at (-3,-4.4) {$\bar c$};
\node at (-4,-4.4) {$\bar d$};
\node at (1.5,0.5) {$\emptyset_1$};
\node at (1.5,-0.5) {$\emptyset_1$};
\node at (1.5,-1.5) {$\emptyset_1$};
\node at (1.5,-2.5) {$\emptyset_1$};
\node at (1.5,-3.5) {$\emptyset_1$};
\node at (2.5,-1.5) {$\emptyset_1$};
\node at (2.5,-2.5) {$\emptyset_1$};
\node at (2.5,-3.5) {$\emptyset_1$};
\node at (3.5,0.5) {$\emptyset_1$};
\node at (3.5,-0.5) {$\emptyset_1$};
\node at (3.5,-1.5) {$\emptyset_1$};
\node at (3.5,-2.5) {$\emptyset_1$};
\node at (3.5,-3.5) {$\emptyset_1$};
\node at (0.5,-1.5) {$\emptyset_1$};
\node at (-0.5,-1.5) {$\emptyset_1$};
\node at (0.5,-3.5) {$\emptyset_1$};
\node at (-0.5,-3.5) {$\emptyset_1$};

\node at (-1.5,-0.5) {$\emptyset_1$};
\node at (-1.5,0.5) {$\emptyset_1$};
\node at (-1.5,1.5) {$\emptyset_1$};
\node at (-1.5,2.5) {$\emptyset_1$};
\node at (-1.5,3.5) {$\emptyset_1$};
\node at (-2.5,1.5) {$\emptyset_1$};
\node at (-2.5,2.5) {$\emptyset_1$};
\node at (-2.5,3.5) {$\emptyset_1$};
\node at (-3.5,-0.5) {$\emptyset_1$};
\node at (-3.5,0.5) {$\emptyset_1$};
\node at (-3.5,1.5) {$\emptyset_1$};
\node at (-3.5,2.5) {$\emptyset_1$};
\node at (-3.5,3.5) {$\emptyset_1$};
\node at (-0.5,1.5) {$\emptyset_1$};
\node at (0.5,1.5) {$\emptyset_1$};
\node at (-0.5,3.5) {$\emptyset_1$};
\node at (0.5,3.5) {$\emptyset_1$};

\node at (1.5,2.5) {$\emptyset_2$};
\node at (1.5,3.5) {$\emptyset_2$};
\node at (2.5,1.5) {$\emptyset_2$};
\node at (2.5,2.5) {$\emptyset_2$};
\node at (2.5,3.5) {$\emptyset_2$};
\node at (3.5,1.5) {$\emptyset_2$};
\node at (3.5,2.5) {$\emptyset_2$};
\node at (3.5,3.5) {$\emptyset_2$};
\node at (-1.5,-2.5) {$\emptyset_2$};
\node at (-1.5,-3.5) {$\emptyset_2$};
\node at (-2.5,-1.5) {$\emptyset_2$};
\node at (-2.5,-2.5) {$\emptyset_2$};
\node at (-2.5,-3.5) {$\emptyset_2$};
\node at (-3.5,-1.5) {$\emptyset_2$};
\node at (-3.5,-2.5) {$\emptyset_2$};
\node at (-3.5,-3.5) {$\emptyset_2$};

\node at (-2.5,0.5) {$A$};
\node at (-2.5,-0.5) {$A$};
\node at (2.5,0.5) {$A$};
\node at (2.5,-0.5) {$A$};
\end{tikzpicture}
\caption{Diagram analysis for $21\bar3\bar4$}
\label{fig:21bar3bar4}
\end{figure}
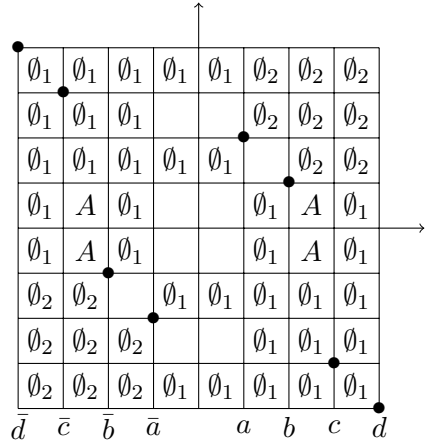
Let $J=\{0,1,2,\ldots,n-2\}\subset S$. If $h(w_J)\leq2$, then $h(w)\leq h(w_J)\leq 2$ and we are done, so assume $h(w_J)\geq3$. Recall that group multiplication is an injection $[e,w^J]^J\times [e,w_{J}]\hookrightarrow [e,w]$. Let \[U=[e,w]\setminus \left([e,w^J]^J\times[e,w_J]\right).\] In other words, $u\in U$ if and only if $u\leq w$ and $u_J\nleq w_J$. 

Since $w(n)=\bar n$, $L^J(w^J)=L^J(w_0^J)$ is palindromic, as the Poincar\'{e} polynomial of $G/P_J$, and thus $h(L^J(w^J)L(w_J))=h(w_J)\geq3$. For any $u\in U$, we have $u_J\not \leq w_J$, and as $w_J(n{-}1)=\overline{n{-}1}$, Claim~\ref{claim:greater-than-rank2} implies that $\ell(u_J)\geq3$, and thus $\ell(u)\geq3$. Now consider an explicit element $v=ws_{n-2}s_{n-1}$ with length $\ell(w)-2$. Note that the index $\pm(n{-}2)$ is either $\pm b$ or inside the region $A$. As $v(\bar n)=w(\overline{n{-}2})\leq n-4$, we have $\ell(v^J)\leq 2n-6$ and $\ell(v_J)=\ell(v)-\ell(v^J)\geq \ell(w)-2n+4$. At the same time, $\ell(w_J)=\ell(w)-\ell(w^J)=\ell(w)-2n+2$. Via this calculation on length, $v_J\nleq w_J$ so $v\in U$. Consider the Poincar\'{e} polynomial
\[L(w)=L^J(w^J)L(w_J)+\sum_{u\in U}q^{\ell(u)}.\]
We have shown that $[q^{i}]L^J(w^J)L(w_J)=[q^{\ell(w)-i}]L^J(w^J)L(w_J)$ for $i\leq 2$ and that there are no elements of length $2$ in $U$ but that there are elements of length $\ell(w)-2$ in $U$. Together this implies that $h(w)\leq 2$. 
\end{proof}

\begin{defin}
Define the \emph{magnitude} $\mg(w)$ as the smallest $b>0$ such that $w$ has an occurrence of $\pm 1 \bar{3} \bar{2}$ with values $a \bar{c} \bar{b}$.
\end{defin}

\begin{prop}
\label{prop:132-bound}
Suppose $w \in \D_n$ contains $\pm 1 \bar{3} \bar{2}$ and avoids $4231$, then $h(w) \leq \mg(w)-1$.
\end{prop}
\begin{proof}
Suppose that $w \in \D_n$ avoids $4231$ and contains $\pm 1 \bar{3} \bar{2}$. Suppose the occurrence of $\pm 1 \bar{3} \bar{2}$ realizing $\mg(w)$ is a $1 \bar{3} \bar{2}$ (the other case being exactly analogous); let this occurrence occupy positions $0<i<j<k$ with values $a \bar{c} \bar{b}$ where $0<a<b<c$, so $\mg(w)=b$. Furthermore, among such occurrences, assume we have chosen one with $c$ minimal.
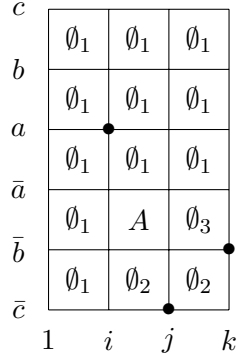
\begin{figure}[h!]
\centering
\begin{tikzpicture}[scale=0.8]
\draw[step=1.0] (0,0) grid (3,5);
\node at (-0.5,0) {$\bar{c}$};
\node at (-0.5,1) {$\bar{b}$};
\node at (-0.5,2) {$\bar{a}$};
\node at (-0.5,3) {$a$};
\node at (-0.5,4) {$b$};
\node at (-0.5,5) {$c$};
\node at (0,-0.5) {$1$};
\node at (1,-0.5) {$i$};
\node at (2,-0.5) {$j$};
\node at (3,-0.5) {$k$};

\node at (1,3) {$\bullet$};
\node at (2,0) {$\bullet$};
\node at (3,1) {$\bullet$};

\node at (1.5,1.5) {$A$};

\node at (0.5,0.5) {$\emptyset_1$};
\node at (0.5,1.5) {$\emptyset_1$};
\node at (0.5,2.5) {$\emptyset_1$};
\node at (1.5,0.5) {$\emptyset_2$};
\node at (1.5,2.5) {$\emptyset_1$};
\node at (0.5,3.5) {$\emptyset_1$};
\node at (0.5,4.5) {$\emptyset_1$};
\node at (1.5,3.5) {$\emptyset_1$};
\node at (1.5,4.5) {$\emptyset_1$};
\node at (2.5,0.5) {$\emptyset_2$};
\node at (2.5,1.5) {$\emptyset_3$};
\node at (2.5,2.5) {$\emptyset_1$};
\node at (2.5,3.5) {$\emptyset_1$};
\node at (2.5,4.5) {$\emptyset_1$};
\end{tikzpicture}
\caption{The diagram of the element $w \in \D_n$ considered in the proof of Proposition~\ref{prop:132-bound}.}
\label{fig:132-bound}
\end{figure}

In the diagram for $w$ shown in Figure~\ref{fig:132-bound}, the regions marked $\emptyset$ are empty for the following reasons.
\begin{itemize}
    \item $\emptyset_1$: if there were a $\bullet$ here, then $w$ would contain $4231$, after taking into account the negative positions of $w$.
    \item $\emptyset_2$: if there were a $\bullet$ here, then $c$ would not be minimal, as assumed.
    \item $\emptyset_3$: if there were a $\bullet$ here, then $a \bar{c} \bar{b}$ would not realize $\mg(w)$.
\end{itemize}
The region $A$ may be nonempty, but, by $4231$ avoidance, can contain only a decreasing sequence.

Let $J_1=\{0,1,2,\ldots,k-1\}$ and $J_2=\{0,1,2,\ldots,c-1\}$ (using the numbering of the Dynkin diagram from Figure~\ref{fig:dynkin-diagrams}). Let $u=\prescript{}{J_2}{(w_{J_1})}$. By Theorem~\ref{thm:h-increases-in-parabolic} and (\ref{eq:h-increases-left-parabolic}) we have $h(w) \leq h(u)$, and by the above analysis of the diagram in Figure~\ref{fig:132-bound}, $u$ has window $[u(1), \ldots, u(m)]=[\pm 1, \bar{2}, \bar{3}, \ldots, \overline{m{-}2}, \overline{m}, \overline{m{-}1}].$ That is, $u=w_0(\D_m) s_{m-1}$. Here $m=3+|A| \leq b+1$.

We work now within $W=\D_m$. Let $K=\{0,1,2,\ldots,m-2\}$, then the corresponding parabolic decomposition has $u^K=s_{m-1} w_0^K$ and $u_K=w_0(K)$ and is a BP-decomposition. Since $L(w_0(K))$ is palindromic (as the Poincar\'{e} polynomial of the smooth Type $D_{m-1}$ flag variety), we have $h(u)=h(L^K(u^K))$ by Lemma~\ref{lem:min-of-h}. The variety $G/P_K=X^K(w_0^K)$ is a $(2m-2)$-dimensional quadric with Poincar\'{e} polynomial \[L^K(w_0^K)=(1+q+\cdots + q^{2m-2})+q^{m-1}\] (see e.g. \cite{Thomas-Yong}). The Schubert variety $X^K(u^K)$ is the closure of the unique complex $(2m-3)$-dimensional cell in $G/P_K$ and thus we have $L^K(u^K)=(1+q+\cdots+q^{2m-3})+q^{m-1}$. We conclude, as desired, that \[h(w) \leq h(u) = h(L^K(u^K))=m-2 \leq b-1 = \mg(w)-1.\]
\end{proof}

\begin{prop}
\label{prop:type-D-avoids-all-but-3412}
Let $W=\D_n$ for $n \geq 5$, let $J=S \setminus \{1\}, J'=S \setminus \{0\}, K=S \setminus \{n-1\}$, and suppose $w \in \D_n$ is singular, but satisfies:
\begin{itemize}
    \item[(i)] $w$ avoids $4231, \pm 1 \bar{3} \bar{2}, \pm 1 2 \bar{3}, \pm 1 4 \bar{3} 2,$ and neither $w$ nor $\varepsilon_D(w)$ contains any occurrences of $3412$ of height one,
    \item[(ii)] $w_J,w_{J'},w_K,\prescript{}{J}{w},\prescript{}{J'}{w},\prescript{}{K}{w}$ are smooth.
\end{itemize}
Then $w=u \coloneqq [n, 2, \bar{3}, \bar{4}, \ldots, \overline{n{-}1}, \pm 1]$ or $w=\varepsilon_D(u)$.
\end{prop}
\begin{proof}

Call any occurrence $cdab=w(i)w(j)w(k)w(\ell)$ of $3412$ with $d=w(j)=n$ and $\ell=n$ \emph{justified}.

\begin{claim}
There is a justified occurrence of $3412$ in $w$.
\end{claim}
\begin{proof}
Since $w$ is singular, but avoids the specified patterns, $w$ must contain $3412$ by Theorem~\ref{thm:bp-smoothness}. Since $w_K$ and $\prescript{}{K}{w}$ are smooth, any occurrence of $3412$ must use an index $\pm n$ as well as a value $\pm n$. Any such occurrence is related by the symmetry $w(-i)=-w(i)$ to a justified occurrence or to one of the form $w(-n)w(i)w(j)w(k)=cnab$. In this case, if $-c<a$, then $nab\bar{c}$ is an occurrence of $4231$ in $w$, and if $-k>i$, then $cab\bar{n}$ is an occurrence of $4231$. Since this contradicts our assumption, we must instead have $-c>a$ and $-k<i$; in this case $\bar{b}na\bar{c}$ is a justified occurrence of $3412$. 
\end{proof}

\begin{claim}
\label{claim:j-b-pm-1}
Let $cnab=w(i)w(j)w(k)w(n)$ be a justified occurrence of $3412$ in $w$. Furthermore assume that $c$ is minimal among such occurrences. Then $j= \pm1$ and $b=\pm 1$.
\end{claim}
\begin{proof}
Suppose that $b>1$, and let $m \in \{1,\ldots, n\}$ be the index such that $|w(m)|=1$. If $j<m$, then $cnw(m)b$ is an occurrence of $3412$ in $\prescript{}{J}{w}$ or $\prescript{}{J'}{w}$, a contradiction. Thus we must have $0<m<j$. Now, since $w$ does not contain a height one $3412$, we know $c>b+1$. Let $m'=w^{-1}(b+1)$. By the minimality of $c$, we cannot have $m' < j$. Nor can we have $m'>k$, or we would have a $3412$ in $w_K$. But now we have a $4231$ in $w$ in positions $imm'k$, a contradiction.

So suppose instead that $b<-1$ and again let $m \in \{1,\ldots, n\}$ be the index such that $|w(m)|=1$. If $k<m$, then there is a $3412$ in positions $ijkm$ unless $c=\pm 1$ and if $k>m$, then we have $\pm1 \bar{3} \bar{2}$ in positions $mkn$. Neither of these is allowed by assumption, so we must have $k<m$ and $c=\pm 1$. Since $w$ and $\varepsilon_D(w)$ have $\mh > 1$ we have $b \leq -3$. Consider $m'=w^{-1}(b+1)$. We have $m'>j$, otherwise positions $m'jkn$ would contain a $3412$ of height one. We must also have $m'<k$ or we would have a $3412$ in $\prescript{}{K}{w}$. We cannot have $0<m'<k$ or we would have $\bar{1}\bar{3}\bar{2}$ with values $(b+1)ab$. Thus we must have $j<m'<0$. But now we have an occurrence of $4231$ in positions $jm'mn$, a contradiction. Thus $b=\pm 1$. Applying the same argument to $w^{-1}$ shows that $j = \pm 1$ as well.
\end{proof}

In the setting of Claim~\ref{claim:j-b-pm-1}, suppose $j=1$ (otherwise, apply $\varepsilon_D$) and assume we have chosen a justified occurrence with $c$ minimal and with $a$ maximal among occurrences with $w(i)=c$.  We now argue that $w=u$. We must have $i<-1$ and $a<-1$ or one of $w_J,w_{J'},\prescript{}{J}{w},\prescript{}{J'}{w}$ would contain the occurrence of $3412$ and thus not be smooth, contrary to our hypotheses. We must also have $|a|>c$, or $|a| n \bar{c} b$ would contradict the minimality of $c$ in $cnab$. Consider the permutation diagram for $w$, drawn in Figure~\ref{fig:avoid-all-but-3412-proof}.

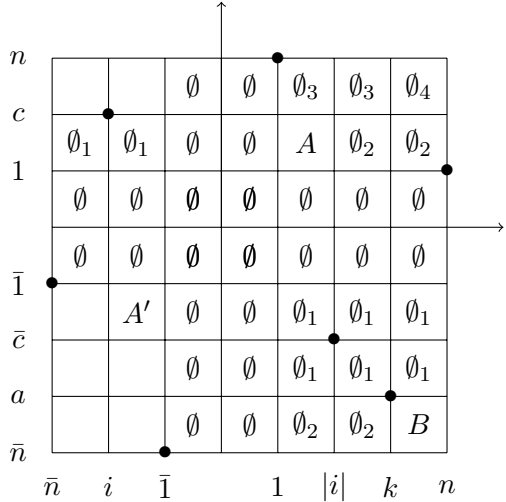
\begin{figure}[h!]
\centering
\begin{tikzpicture}[scale=0.75]
\draw[step=1.0] (-4,-4) grid (3,3);
\node at (-4,-1) {$\bullet$};
\node at (-3,2) {$\bullet$};
\node at (-2,-4) {$\bullet$};
\node at (0,3) {$\bullet$};
\node at (1,-2) {$\bullet$};
\node at (2,-3) {$\bullet$};
\node at (3,1) {$\bullet$};
\draw[->](3,0)--(4,0);
\draw[->](-1,0)--(-1,4);
\node at (0,-4.6) {$1$};
\node at (1,-4.6) {$|i|$};
\node at (2,-4.6) {$k$};
\node at (3,-4.6) {$n$};
\node at (-2,-4.6) {$\bar{1}$};
\node at (-3,-4.6) {$i$};
\node at (-4,-4.6) {$\bar{n}$};
\node at (-4.6,-4) {$\bar{n}$};
\node at (-4.6,-3) {$a$};
\node at (-4.6,-2) {$\bar{c}$};
\node at (-4.6,-1) {$\bar{1}$};
\node at (-4.6,1) {$1$};
\node at (-4.6,2) {$c$};
\node at (-4.6,3) {$n$};

\node at (-3.5,0.5) {$\emptyset$};
\node at (-2.5,0.5) {$\emptyset$};
\node at (-1.5,0.5) {$\emptyset$};
\node at (-0.5,0.5) {$\emptyset$};
\node at (0.5,0.5) {$\emptyset$};
\node at (1.5,0.5) {$\emptyset$};
\node at (2.5,0.5) {$\emptyset$};
\node at (-3.5,-0.5) {$\emptyset$};
\node at (-2.5,-0.5) {$\emptyset$};
\node at (-1.5,-0.5) {$\emptyset$};
\node at (-0.5,-0.5) {$\emptyset$};
\node at (0.5,-0.5) {$\emptyset$};
\node at (1.5,-0.5) {$\emptyset$};
\node at (2.5,-0.5) {$\emptyset$};

\node at (-0.5,2.5) {$\emptyset$};
\node at (-0.5,1.5) {$\emptyset$};
\node at (-0.5,0.5) {$\emptyset$};
\node at (-0.5,-0.5) {$\emptyset$};
\node at (-0.5,-1.5) {$\emptyset$};
\node at (-0.5,-2.5) {$\emptyset$};
\node at (-0.5,-3.5) {$\emptyset$};
\node at (-1.5,2.5) {$\emptyset$};
\node at (-1.5,1.5) {$\emptyset$};
\node at (-1.5,0.5) {$\emptyset$};
\node at (-1.5,-0.5) {$\emptyset$};
\node at (-1.5,-1.5) {$\emptyset$};
\node at (-1.5,-2.5) {$\emptyset$};
\node at (-1.5,-3.5) {$\emptyset$};

\node at (-3.5,1.5) {$\emptyset_1$};
\node at (-2.5,1.5) {$\emptyset_1$};
\node at (2.5,1.5) {$\emptyset_2$};
\node at (1.5,1.5) {$\emptyset_2$};
\node at (0.5,2.5) {$\emptyset_3$};
\node at (1.5,2.5) {$\emptyset_3$};
\node at (0.5,1.5) {$A$};
\node at (-2.5,-1.5) {$A'$};
\node at (2.5,2.5) {$\emptyset_4$};
\node at (0.5,-1.5) {$\emptyset_1$};
\node at (1.5,-1.5) {$\emptyset_1$};
\node at (2.5,-1.5) {$\emptyset_1$};
\node at (0.5,-2.5) {$\emptyset_1$};
\node at (1.5,-2.5) {$\emptyset_1$};
\node at (2.5,-2.5) {$\emptyset_1$};
\node at (0.5,-3.5) {$\emptyset_2$};
\node at (1.5,-3.5) {$\emptyset_2$};
\node at (2.5,-3.5) {$B$};
\end{tikzpicture}
\caption{The diagram analysis for the justified occurrence $w(i)w(1)w(k)w(n)=cna {\pm1}$ of $3412$ in $w$. The diagram is drawn with $|i|<k$ and $b=1$, but if $|i|>k$ or $b=-1$ the analysis is not materially affected.}
\label{fig:avoid-all-but-3412-proof}
\end{figure}

The indicated regions of the permutation diagram for $w$ shown in Figure~\ref{fig:avoid-all-but-3412-proof} are empty for the following reasons:
\begin{itemize}
    \item $\emptyset_1$: if there were a $\bullet$ here, then the minimality of $c$ would be contradicted.
    \item $\emptyset_2$: if there were a $\bullet$ here, then $w_K$ would contain $3412$.
    \item $\emptyset_3$: if there were a $\bullet$ here, then $\prescript{}{K}{w}$ would contain $3412$.
\end{itemize}

Now, consider the region $A$. Since all other regions in the same row are known to be empty, there must be $c-2$ $\bullet$'s inside $A$. The region $A'$ has the same number of $\bullet$'s, since it is obtained from $A$ by negating indices and values. Thus if $c-2>1$, there is a $4231$ pattern in $w$ with values $c \bar{3} 2 a$. This violates our hypotheses, so $c \leq 3$. But we have also assumed that $w$ contains no $3412$ of height one, so $c=3$. We now see:
\begin{itemize}
    \item $\emptyset_4$: if there were a $\bullet$ here, then $w$ would contain $4231$.
\end{itemize}

Finally, region $B$ must contain a decreasing sequence, since any ascent within $B$ would lead to a $3412$ in $w_K$. We conclude, as desired, that
\[w= [n, 2, \bar{3}, \bar{4}, \ldots, \overline{n{-}1}, \pm 1]=u.\] 
\end{proof}

We are now ready to complete the proof of Theorem~\ref{thm:simply-laced-upper-bound}, resolving Conjecture~\ref{conj:bp-conjecture}.

\begin{proof}[Proof of Theorem~\ref{thm:simply-laced-upper-bound}]   
First suppose $G=\SL_{r+1}$, and let $w \in W(A_r)=\S_{r+1}$ such that $X_w$ is singular. By Theorem~\ref{thm:bp-smoothness}, $w$ contains $4231$ or $3412$. If $w$ contains $4231$, then $h(w)=1$ by Proposition~\ref{prop:h-equals-1}. Otherwise $w$ avoids $4231$ and contains $3412$, so $h(w) \leq \mh(w)$ by Lemma~\ref{lem:typeA-3412-upper}. It is clear by definition that $\mh(w) \leq r-2$ for any $w$, so we are done.

Now suppose $G=\SO_{2r}$ for $r \geq 5$, and let $w \in W(D_r)=\D_r$. Suppose by induction that the claim is true for $G=\SO_{2r'}$ for $r' < r$ (the base case $r'=4$ is covered by the computations in \cite{Billey-Postnikov}). If $w$ contains $4231$, then $h(w) \leq 2 \leq r-2$ by Proposition~\ref{prop:type-D-4231-upper-bound}, so we may assume that $w$ avoids $4231$. Then by Proposition~\ref{prop:132-bound}, if $w$ contains $\pm 1 \bar{3} \bar{2}$ we have $h(w) \leq \mg(w) \leq r-2$. If $w$ contains any of the patterns from Proposition~\ref{prop:h-equals-1}, then $h(w)=1 \leq r-2$. Let $J=S \setminus \{2\}, J'=S \setminus \{0\}, K=S \setminus \{r-1\}$; if any of $w_J,w_{J'},w_K,\prescript{}{J}{w},\prescript{}{J'}{w},\prescript{}{K}{w}$ is singular, then by the type $A$ result, or by the induction hypothesis, we have $h(w) \leq r-3$. Finally, if $w$ does not fall into any of the above cases, then $w$ satisfies the hypotheses (i) and (ii) of Proposition~\ref{prop:type-D-avoids-all-but-3412}, so $w=u \coloneqq [r, 2, \bar{3}, \bar{4}, \ldots, \overline{r{-}1}, \pm 1]$ or $w=\varepsilon_D(u)$.

We will now compute $h(u)=h(\varepsilon_D(u))$; suppose for convenience that $r$ is even, the other case being exactly analogous. Let $I=\{1,2\ldots,r-2\}$, then we have $u_I=w_0(I)$ is the longest element of $\S_{r-1}$, so $h(u_I)=\infty$. Thus we need to compute $h(L^I(u^I))$ with $u^I=[\overline{r{-}1}, \ldots, \bar{4},\bar{3},2,r,\bar{1}]$. Notice $\ell(u^I)=N \coloneqq \frac{1}{2}(r^2-3r+4)$ with reduced word:
\[
s_0(s_2s_0)(s_3s_2s_1) \cdots (s_{r-4}s_{r-5}\cdots s_3s_2s_0)(s_{r-3}\cdots s_3s_2s_1)(s_{r-2}\cdots s_3s_2s_0)s_{r-1}.
\]

We claim that $L^I(u^I)=1+2q+3q^2+ \cdots + aq^{N-2}+2q^{N-1}+q^N$, with $a \geq 4$, so that $h(u)=h(L^I(u^I))=2 < r-2$. Indeed, the elements of length one in $[e,u^I]^I$ are $\{s_0, s_{r-1}\}$, the elements of length two are $\{s_0s_{r-1}, s_2s_0, s_{r-2}s_{r-1}\}$, and the elements of length $N-1$ are $\{s_0u^I, s_2u^I\}$. Consider the four elements $z_1=s_0s_2 u^I, z_2=s_2s_0u^I, z_3=s_0u^Is_{r-1}, z_4=s_3s_2u^I$. It is easy to check for $i=1,2,3,4$ that $\ell(z_i)=N-2$, that $z_i \leq u^I$ (by Theorem~\ref{thm:subword-bruhat}), and that $z_i \in W^I$; thus $a \geq 4$ as desired.
\end{proof}

\section{Exact formula when $G=\SL_n$}
\label{sec:lower-bound}
For $w\in\S_n$, we have proved one direction of Theorem~\ref{thm:type-A-formula}, the upper bound, in Section~\ref{sub:typeA-upper-bound}. In this section, we establish the other direction.
\begin{lemma}\label{lem:typeA-3412-lower}
Suppose that $w\in\S_n$ avoids $4231$ and contains $3412$. Then $h(w)\geq\mh(w)$.
\end{lemma}
\begin{proof}
We use induction on $n$. The statement is true for $n=4$, where $h(3412)=\mh(3412)=1$.

In the permutation diagram for $w$, consider the points $(1,w(1))$ and $(w^{-1}(1),1)$ and the following three regions:
\begin{align*}
A=&\{(i,w(i))\:|\: 1<i<w^{-1}(1),\ 1<w(i)<w(1)\},\\
B=&\{(i,w(i))\:|\: i>w^{-1}(1),\ 1<w(i)<w(1)\},\\
C=&\{(i,w(i))\:|\: 1<i<w^{-1}(1),\ w(i)>w(1)\}
\end{align*}
as shown in Figure~\ref{fig:3412-lower-bound}.
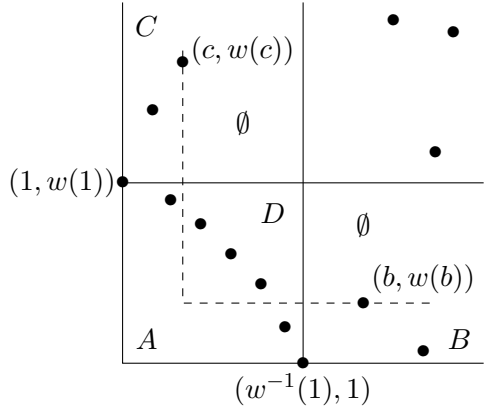
\begin{figure}[h!]
\centering
\begin{tikzpicture}[scale=0.8]
\draw(0,6)--(0,0)--(6,0);
\draw(3,0)--(3,6);
\draw(0,3)--(6,3);
\node at (0,3) {$\bullet$};
\node at (3,0) {$\bullet$};
\draw[dashed] (1,5.2)--(1,1)--(5.2,1);
\node at (1,5) {$\bullet$};
\node at (2,5.2) {$(c,w(c))$};
\node at (4,1) {$\bullet$};
\node at (5,1.4) {$(b,w(b))$};
\node at (0.8,2.7) {$\bullet$};
\node at (1.3,2.3) {$\bullet$};
\node at (1.8,1.8) {$\bullet$};
\node at (2.3,1.3) {$\bullet$};
\node at (2.7,0.6) {$\bullet$};
\node at (0.4,0.4) {$A$};
\node at (5.6,0.4) {$B$};
\node at (0.4,5.6) {$C$};
\node at (0.5,4.2) {$\bullet$};
\node at (2,4) {$\emptyset$};
\node at (4,2.3) {$\emptyset$};
\node at (5,0.2) {$\bullet$};
\node at (5.2,3.5) {$\bullet$};
\node at (5.5,5.5) {$\bullet$};
\node at (4.5,5.7) {$\bullet$};
\node at (-1,3) {$(1,w(1))$};
\node[below] at (3,0) {$(w^{-1}(1),1)$};
\node at (2.5,2.5) {$D$};
\end{tikzpicture}
\caption{The permutation diagram of $w$ and the indicated regions, as discussed in the proof of Lemma~\ref{lem:typeA-3412-lower}.}
\label{fig:3412-lower-bound}
\end{figure}
Since $w$ avoids $4231$, the elements in $A$ must be in decreasing.

If $B=\emptyset$, letting $J=\{2,3,\ldots,n-1\}$ we have $w^J=m123\cdots n=s_{m-1}s_{m-2}\cdots s_2s_1$ where $m=w(1)$. At the same time, when $B=\emptyset$, $w_J(1)=1$ and $2,3,\ldots,m$ appear in reversed order in $w_J$, meaning that $s_2,\ldots,s_{m-1}\in D_L(w_J)$. Thus, $\supp(w^J)\cap J\subset D_L(w_J)$, and $w=w^Jw_J$ is a BP-decomposition, so $L(w)=L^J(w^J)L(w_J)$ by Theorem~\ref{thm:bp-decomposition-fiber-bundle}. As $L^J(w^J)=1+q+\cdots+q^{m-1}$ is palindromic, by Lemma~\ref{lem:min-of-h} we have $h(w)=h(w_J)$. Moreover, $B=\emptyset$ implies that no $3412$ uses $(1,w(1))$ so $\mh(w)=\mh(w_J)$. By induction, $h(w)=h(w_J)\geq\mh(w_J)=\mh(w)$. Similarly, if $C=\emptyset$, we can apply the induction hypothesis to $w^{-1}$.

From now on, assume that both $B$ and $C$ are nonempty. Let $(c,w(c))\in C$ with $c$ maximal, and let $(b,w(b))\in B$ with $w(b)$ maximal. The $3412$ occurrence at indices $1<c<w^{-1}(1)<b$ has content $|D|+1$, where \[D=\{(i,w(i))\:|\: c<i<w^{-1}(1),\ w(b)<w(i)<w(1)\}\subset A.\] 
In fact, by construction, $D=\{(p,q),(p+1,q-1),\ldots,(p+d-1,q-d+1)\}$ with $d=|D|$, $p=c+1$, and $q-d+1=w(b)+1$. 

Let $J=\{2,3,\ldots,n-1\}$. For the rest of the proof, we will carefully study the set \[U:=[e,w]\setminus \left([e,w^J]^J\cdot[e,w_{J}] \right) \neq\emptyset.\] In other words, $u\in U$ if $u\leq w$ but $u_J\nleq w_J$. The key is to show that any $u\in U$ cannot be too close to $e$, or too close to $w$.

Let us recall one version of the tableau criterion for the Bruhat order (see Theorem 2.6.3 of \cite{Bjorner-Brenti}). For two sets $X,Y\subset[n]$ with $|X|=|Y|$, we say $X\leq Y$ in \emph{Gale order} if after sorting $X=\{x_1<\cdots<x_k\}$ and $Y=\{y_1<\cdots<y_k\}$, we have $x_i\leq y_i$ for $i=1,\ldots,k$. For $u\in\S_n$, write $u[i:j]$ for $\{u(i),u(i+1),\ldots,u(j)\}$. Then $u\leq v$ if and only if $u[1:k]\leq v[1:k]$ for all $k=1,2,\ldots,n$. 
\begin{claim}\label{claim:typeA-lower-top-elements}
For $u\in U$, $\ell(w)-\ell(u)\geq|D|+1$.
\end{claim}
\begin{proof}
If $D=\emptyset$, then clearly $\ell(w)-\ell(u)\geq1$ since $u\leq w$ and $w\notin U$. 

For $D\neq\emptyset$, let $w'=s_{w(1)-1}w\lessdot w$. We first show that $u\leq w'$. Pictorially, $w'$ is obtained from $w$ by swapping $(1,w(1))$ with the permutation entry closest to $(1,w(1))$ in $A$. Assume to the contrary that $u\nleq w'$, then there exists $r$ such that $u[1:r]\nleq w'[1:r]$. As $w'[1:k]=w[1:k]$ for $k\geq w^{-1}(w(1)-1)$, we need $r<w^{-1}(w(1)-1)$. By construction, after sorting from smallest to largest, $w'[1:r]=\{w(1)-1<\cdots\}$ and $w[1:r]=\{w(1)<\cdots\}$ only differ at the smallest value. In order for $u[1:r]\nleq w'[1:r]$ and $u[1:r]\leq w[1:r]$, the smallest value in $u[1:r]$ is at least $w(1)$. At the same time, $u(1)\leq w(1)$ so we must have $u(1)=w(1)$. However, $u(1)=w(1)$ and $u\leq w$ implies that $u_J\leq w_J$ for $J=\{2,\ldots,n-1\}$, contradicting $u\in U$. 

As $w'_J=w_J$, we now have $u\leq w'$ and $u_J\nleq w_J'$. Going through the same argument for $w'$ instead of $w$, we arrive at $u\leq v:=s_{w(b)+1}\cdots s_{w(1)-2}s_{w(1)-1}w$, after swapping $(1,w(1))$ with every permutation entry in $D$. Now $\ell(w)-\ell(v)=w(1)-w(b)-1\geq|D|$. Since $u\leq v$ and $u_J\nleq v_J=w_J$, we have $u\neq v$ and $\ell(v)-\ell(u)\geq1$. Thus, $\ell(w)-\ell(u)\geq|D|+1$ as desired.
\end{proof}

\begin{claim}\label{claim:typeA-lower-bottom-elements}
For $u\in U$, $\ell(u)\geq |D|+2$.
\end{claim}
\begin{proof}
Since $u_J\nleq w_J$, there exists $r$ such that $u_J[1:r]\nleq w_J[1:r]$. Note $u_J(1)=w_J(1)=1$. At the same time, $u_J[1:r]\leq u[1:r]\leq w[1:r]$. First, $r\geq p+d-1$, the index of the last entry in $D$, because for $k\leq p+d-1$, $w_J[1:k]$ and $w[1:k]$ differ only in the smallest value, and with $u_J(1)=1$, we cannot have both $u_J[1:k]\leq w[1:k]$ and $u_J[1:k]\nleq w_J[1:k]$. 

Recall that to obtain the set $w_J[1:r]$ from $w[1:r]$, we can first remove $w(1)$ from $w[1:r]$, add $1$ to all the values less than $w(1)$, and insert $1$. Suppose that there are $m$ values that are at most $w(b)$ in $w[1:r]$, denoted $a_1,\ldots,a_m$. We have the inequality $r-m\geq p+d-1\geq 1+|D|+|C|$. Correspondingly, $a_1+1,a_2+1,\ldots,a_m+1\in w_J[2:r]$ and we also have $w_J(1)=1$. The values greater than $w(1)$ stay the same in $w[1:r]$ and $w_J[1:r]$. Also, the consecutive interval $\{w(b)+1,\ldots,w(1)\}$ is contained in $w[1:r]$, so $\{w(b)+2,\ldots,w(1)\}\subset w_J[1:r]$. Thus, regardless of whether $r\geq b$ or not, $w_J[1:r]$ and $w[1:r]$ can only differ at the smallest $m+1$ values. Since $u_J[1:r]\nleq w_J[1:r]$ and $u_J[1:r]\leq w[1:r]$, we know that $u_J[1:r]$ cannot contain all of $1,2,\ldots,m+1$. Let $j\leq m+1$ be the smallest number that $u_J[1:r]$ does not contain. There are at least $r-m$ values in $u_J[1:r]$ that are greater than $j$, creating at least $r-m$ inversions with $j$. As a result, $\ell(u)\geq\ell(u_J)\geq r-m\geq 1+|D|+|C|\geq|D|+2$ as desired.

As an example for visualization, we can take \begin{align*}
w=&9,11,8,12,7,6,5,3,1,4,14,2,10,13,\\
w_J=&1,11,9,12,8,7,6,4,2,5,14,3,10,13,
\end{align*} as shown in Figure~\ref{fig:3412-lower-bound}. Take $r=9$ with $m=2$, \begin{align*}
\mathrm{sort}(w[1:r])=&1,3,5,6,7,8,9,11,12\\
\mathrm{sort}(w_J[1:r])=&1,2,4,6,7,8,9,11,12.
\end{align*}
Seeing that they agree after the third position, $u_J[1:r]$ cannot contain all of $1,2,3$. The rest follows from a counting argument. 
\end{proof}

In the previous two claims, we established that any $u\in U$ satisfies \[|D|+2\leq \ell(u)\leq \ell(w)-|D|-1\] where $\mh(w)=\mcontent(w)\leq|D|+1$. By the induction hypothesis, $h(w_J)\geq\mh(w_J)=\mcontent(w_J)\geq\mcontent(w)=\mh(w)$ since every $3412$ in $w_J$ is a $3412$ in $w$ with the same content. By Lemma~\ref{lem:min-of-h}, as $L^J(w^J)$ is palindromic, we have $h(L^J(w^J)L(w_J))=h(w_J)\geq\mh(w)$, so the polynomial $L^J(w^J)L(w_J)$ is palindromic up to rank $\mh(w)-1$ (that is, the coefficient of $q^i$ equals the coefficient of $q^{\ell(w)-i}$ for $0\leq i\leq\mh(w)-1$). By the two claims above, every monomial $q^i$ in $L(w)-L^J(w^J)L(w_J)$ is at least $|D|+2\geq\mh(w)+1$ ranks from the bottom and at least $|D|+1\geq\mh(w)$ ranks from the top, meaning that $L(w)$ is still palindromic up to rank $\mh(w)-1$. Therefore, $h(w)\geq\mh(w)$.
\end{proof}

\noindent Theorem~\ref{thm:type-A-formula} now follows from Proposition~\ref{prop:h-equals-1}(i), Lemma~\ref{lem:typeA-3412-upper} and Lemma~\ref{lem:typeA-3412-lower}.

\section*{Acknowledgements}
We are grateful for our very fruitful conversations with Alexander Woo and with Sara Billey. We also thank the anonymous referee for a very careful reading of an earlier version of this manuscript.

\bibliographystyle{plain}
\bibliography{arxiv-v2}
\end{document}